\documentclass[reqno, 12pt]{amsart}
 \usepackage[a4paper, hmargin = 1.25in, vmargin = 1.5in]{geometry}
 \usepackage{amsmath, amssymb, amsfonts}
 \usepackage{amsthm}
 
 \newtheorem{thm}{Theorem}[section]

 \newtheorem{lem}[thm]{Lemma}
 
 \theoremstyle{definition}
 \newtheorem{defn}{Definition}[section]
 \theoremstyle{remark}
 
 \numberwithin{equation}{section}

\begin{document}

\begin{center}

\begin{title}
\title{\bf\Large{{Lyapunov-Type Inequality for a Riemann-Liouville Type Fractional Boundary Value Problem with Robin Boundary Conditions}}}
\end{title}

\vskip 0.25 in

\begin{author}
\author {Jagan Mohan Jonnalagadda\footnote[1]{Department of Mathematics, Birla Institute of Technology and Science Pilani, Hyderabad - 500078, Telangana, India. email: {j.jaganmohan@hotmail.com}}}
\end{author}

\end{center}

\vskip 0.25 in

\noindent{\bf Abstract:} In this article we establish a Lyapunov-type inequality for two-point Riemann-Liouville fractional boundary value problems associated with well-posed Robin boundary conditions. To illustrate the applicability of established result, we deduce criterion for the nonexistence of real zeros of a Mittag-Leffler function.

\vskip 0.25 in


\vskip 0.25 in


\vskip 0.25 in

\section{Introduction}

In 1907, Lyapunov \cite{L} proved a necessary condition for the existence of a nontrivial solution of Hill's equation associated with Dirichlet boundary conditions.

\begin{thm} \cite{L} \label{ODE}
If the boundary value problem
\begin{equation} \label{BVP}
\begin{cases}
y''(t) + p(t)y(t) = 0, \quad a < t < b,\\
y(a) = 0, \quad y(b) = 0, 
\end{cases}
\end{equation}
has a nontrivial solution, where $p : [a, b] \rightarrow \mathbb{R}$ is a continuous function, then
\begin{equation} \label{Lyp}
\int^{b}_{a}|p(s)|ds > \frac{4}{(b - a)}.
\end{equation}
\end{thm}
The Lyapunov inequality \eqref{Lyp} has several applications in various problems related to differential equations. Due to its importance, the Lyapunov inequality has been generalized in many forms. For more details on Lyapunov-type inequalities and their applications, we refer \cite{B,PA,PI,AT,Y1,Y2} and the references therein.

On the other hand, many researchers have derived Lyapunov-type inequalities for various classes of fractional boundary value problems in the recent years. For the first time, in 2013, Ferreira \cite{F1} generalized Theorem \ref{ODE} to the case where the classical second-order derivative in \eqref{BVP} is replaced by an $\alpha^{\text{th}}$-order ($1 < \alpha \leq 2$) Riemann-Liouville type derivative.
\begin{thm} \cite{F1} \label{FDE RL}
If the fractional boundary value problem
\begin{equation*} \label{BVP RL}
\nonumber \begin{cases}
D^{\alpha}_{a}y(t) + p(t)y(t) = 0, \quad a < t < b,\\
y(a) = 0, \quad y(b) = 0, 
\end{cases}
\end{equation*}
has a nontrivial solution, where $p : [a, b] \rightarrow \mathbb{R}$ is a continuous function, then
\begin{equation*} \label{Lyp RL}
\int^{b}_{a}|p(s)|ds > \Gamma(\alpha) \Big{(}\frac{4}{b - a}\Big{)}^{\alpha - 1}.
\end{equation*}
\end{thm}
Here $D^{\alpha}_{a}$ denotes the Riemann-Liouville type $\alpha^{\text{th}}$-order differential operator. In 2014, Ferreira \cite{F2} replaced the Riemann-Liouville type derivative in Theorem \ref{FDE RL} with the Caputo one ${}^{C}D^{\alpha}_{a}$ and obtained the following Lyapunov-type inequality for the resulting problem:
\begin{thm} \cite{F2} \label{FDE C}
If the fractional boundary value problem
\begin{equation*} \label{BVP C}
\begin{cases}
{}^{C}D^{\alpha}_{a}y(t) + p(t)y(t) = 0, \quad a < t < b,\\
y(a) = 0, \quad y(b) = 0, 
\end{cases}
\end{equation*}
has a nontrivial solution, where $p : [a, b] \rightarrow \mathbb{R}$ is a continuous function, then
\begin{equation*} \label{Lyp C}
\int^{b}_{a}|p(s)|ds > \frac{\Gamma(\alpha){\alpha}^{\alpha}}{[(\alpha - 1)(b - a)]^{\alpha - 1}}.
\end{equation*}
\end{thm}
In 2015, Jleli et al. \cite{J3} obtained a Lyapunov-type inequality for two-point Caputo fractional boundary value problems associated with Robin boundary conditions. Recently, Ntouyas et al. \cite{N} presented a survey of results on Lyapunov-type inequalities for fractional differential equations associated with a variety of boundary conditions. This article shows a gap in the literature on Lyapunov-type inequalities for two-point Riemann-Liouville fractional boundary value problems associated with Robin boundary conditions.

In 2016, Dhar et al. \cite{D} derived Lyapunov-type inequalities for two-point Riemann-Liouville fractional boundary value problems associated with fractional integral boundary conditions. This article stresses the importance of choosing well-posed boundary conditions for Riemann-Liouville fractional boundary value problems. 

Motivated by these developments, in this article, we establish a Lyapunov-type inequality for two-point Riemann-Liouville fractional boundary value problems associated with well-posed Robin boundary conditions.

\section{Preliminaries}
Throughout, we shall use the following notations, definitions and known results of fractional calculus \cite{KIL,P1}. Denote the set of all real numbers and complex numbers by $\mathbb{R}$ and $\mathbb{C}$, respectively.

\begin{defn} \cite{KIL}
Let $\alpha > 0$ and $a \in \mathbb{R}$. The $\alpha^{\text{th}}$-order Riemann-Liouville fractional integral of a function $y : [a, b] \rightarrow \mathbb{R}$ is defined by 
\begin{equation}
I^{\alpha}_{a} y(t) = \frac{1}{\Gamma(\alpha)} \int^{t}_{a} (t - s)^{\alpha - 1} y(s)ds, \quad a \le t \le b,
\end{equation}
provided the right-hand side exists. For $\alpha = 0$, define $I^{\alpha}_{a}$ to be the identity map. Moreover, let $n$ denote a positive integer and assume $n - 1 < \alpha \leq n$. The $\alpha^{\text{th}}$-order Riemann-Liouville fractional derivative is defined as 
\begin{equation}
D^{\alpha}_{a} y(t) = D^{n} I^{n - \alpha}_{a} y(t), \quad a \le t \le b,
\end{equation}
where $D^{n}$ denotes the classical $n^{\text{th}}$-order derivative, if the right-hand side exists.
\end{defn}

\begin{defn} \cite{KIL}
We denote by $L(a, b)$ the space of Lebesgue measurable functions $y : [a, b] \rightarrow \mathbb{R}$ for which $$\|y\|_{L} = \int^{b}_{a}|y(t)|dt < \infty.$$
\end{defn}

\begin{defn} \cite{KIL}
We denote by $C[a, b]$ the space of continuous functions $y : [a, b] \rightarrow \mathbb{R}$ with the norm $$\|y\|_{C} = \max_{t \in [a, b]}|y(t)|.$$
\end{defn}

\begin{defn} \cite{KIL}
Let $0 \leq \gamma < 1$, $y : (a, b] \rightarrow \mathbb{R}$ and define $y_{\gamma}(t) = t^{\gamma}y(t)$, $t \in [a, b]$. We denote by $C_{\gamma}[a, b]$ the weighted space of functions $y$ such that $y_{\gamma} \in [a, b]$, and $$\|y\|_{C_{\gamma}} = \max_{t \in [a, b]}|(t - a)^{\gamma}y(t)|.$$
\end{defn}

\begin{lem} \cite{KIL} \label{Power Rule}
If $\alpha \geq 0$ and $\beta > 0$, then
\begin{gather}
\nonumber I^{\alpha}_{a} (t - a)^{\beta - 1} = \frac{\Gamma(\beta)}{\Gamma(\beta + \alpha)}(t - a)^{\beta + \alpha - 1}, \\ \nonumber D^{\alpha}_{a} (t - a)^{\beta - 1} = \frac{\Gamma(\beta)}{\Gamma(\beta - \alpha)}(t - a)^{\beta - \alpha - 1}.
\end{gather}
\end{lem}

\begin{lem} \cite{KIL} \label{Int}
Let $\alpha > \beta > 0$ and $y \in C[a, b]$. Then,
\begin{equation}
\nonumber D^{\beta}_{a}I^{\alpha}_{a}y(t) = I^{\alpha - \beta}_{a}y(t), \quad t \in [a, b].
\end{equation}
\end{lem}

\begin{lem} \cite{B}
Let $\alpha > 0$ and $n$ be a positive integer such that $n - 1 < \alpha \leq n$. If $y \in C(a, b) \cap L(a, b)$, then the unique solution of the fractional differential equation $$D^{\alpha}_{a} y(t) = 0, \quad a < t < b,$$ is $$y(t) = C_{1}(t - a)^{\alpha - 1} + C_{2}(t - a)^{\alpha - 2} + \cdots + C_{n}(t - a)^{\alpha - n},$$ where $C_{i} \in \mathbb{R}$, $i = 1, 2, \cdots, n$.
\end{lem}

\begin{lem} \cite{B} \label{Eq}
Let $\alpha > 0$ and $n$ be a positive integer such that $n - 1 < \alpha \leq n$. If $y \in C(a, b) \cap L(a, b)$, then $$I^{\alpha}_{a} D^{\alpha}_{a} y(t) = y(t) + C_{1}(t - a)^{\alpha - 1} + C_{2}(t - a)^{\alpha - 2} + \cdots + C_{n}(t - a)^{\alpha - n},$$ for some $C_{i} \in \mathbb{R}$, $i = 1, 2, \cdots, n$.
\end{lem}

\section{Robin Boundary Value Problem}

In this section, we obtain a Lyapunov-type inequality for a Robin boundary value problem using the properties of the corresponding Green's function.

\begin{thm} \label{RoF Theorem 1}
Let $1 < \alpha \leq 2$ and $h : [a, b] \rightarrow \mathbb{R}$. The fractional boundary value problem
\begin{equation} \label{RoF FDE 1}
\begin{cases}
D^{\alpha}_{a}y(t) + h(t) = 0, \quad a < t < b, \\ I^{2 - \alpha}_{a}y(a) - D^{\alpha - 1}_{a}y(a) = 0, ~ y(b) + D^{\alpha - 1}_{a}y(b) = 0,
\end{cases}
\end{equation}
has the unique solution 
\begin{equation} \label{RoF Solution}
y(t) = \int^{b}_{a}G(t, s)h(s)ds,
\end{equation}
where
\begin{equation} \label{RoF Green}
G(t, s) = \begin{cases}
\big{(}\frac{(t - a)^{\alpha - 1} + (\alpha - 1)(t - a)^{\alpha - 2}}{A}\big{)}\big{(}\frac{(b - s)^{\alpha - 1}}{\Gamma(\alpha)} + 1\big{)}, \hspace{64pt} a < t \leq s \leq b,\\
\big{(}\frac{(t - a)^{\alpha - 1} + (\alpha - 1)(t - a)^{\alpha - 2}}{A}\big{)}\big{(}\frac{(b - s)^{\alpha - 1}}{\Gamma(\alpha)} + 1\big{)} - \frac{(t - s)^{\alpha - 1}}{\Gamma(\alpha)}, \quad a < s \leq t \leq b.
\end{cases}
\end{equation}
\end{thm}

\begin{proof}
Applying $I^{\alpha}_{a}$ on both sides of \eqref{RoF FDE 1} and using Lemma \ref{Eq}, we have
\begin{equation} \label{Sol 1}
y(t) = -I^{\alpha}_{a}h(t) + C_{1}(t - a)^{\alpha - 1} + C_{2}(t - a)^{\alpha - 2},
\end{equation}
for some $C_{1}$, $C_{2} \in \mathbb{R}$. Applying $I^{2 - \alpha}_{a}$ on both sides of \eqref{Sol 1} and using Lemmas \ref{Power Rule} - \ref{Int}, we get
\begin{equation} \label{Sol 2}
I^{2 - \alpha}_{a}y(t) = -I^{2}_{a}h(t) + C_{1}\Gamma(\alpha)(t - a) + C_{2}\Gamma(\alpha - 1). 
\end{equation}
Applying $D^{\alpha - 1}_{a}$ on both sides of \eqref{Sol 1} and using Lemmas \ref{Power Rule} - \ref{Int}, we get
\begin{equation} \label{Sol 3}
D^{\alpha - 1}_{a}y(t) = -I^{1}_{a}h(t) + C_{1}\Gamma(\alpha). 
\end{equation}
Using $I^{2 - \alpha}_{a}y(a) - D^{\alpha - 1}_{a}y(a) = 0$ in \eqref{Sol 2} and \eqref{Sol 3}, we get 
\begin{equation} \label{Eq 1}
-C_{1}(\alpha - 1) + C_{2} = 0.
\end{equation}
Using $y(b) + D^{\alpha - 1}_{a}y(b) = 0$ in \eqref{Sol 1} and \eqref{Sol 3}, we get 
\begin{equation} \label{Eq 2}
C_{1}\big{[}(b - a)^{\alpha - 1} + \Gamma(\alpha)\big{]} + C_{2}(b - a)^{\alpha - 2} = I^{\alpha}_{a}h(b) + I^{1}_{a}h(b).
\end{equation}
Solving \eqref{Eq 1} and \eqref{Eq 2} for $C_{1}$ and $C_{2}$, we have
$$C_{1} = \frac{1}{A}\int^{b}_{a}\Big{[}\frac{(b - s)^{\alpha - 1}}{\Gamma(\alpha)} + 1\Big{]}h(s)ds,$$ and $$C_{2} = (\alpha - 1)C_{1},$$ where $$A = (b - a)^{\alpha - 1} + (\alpha - 1)(b - a)^{\alpha - 2} + \Gamma(\alpha).$$ Substituting $C_{1}$ and $C_{2}$ in \eqref{Sol 1}, the unique solution of \eqref{RoF FDE 1} is
\begin{align*}
& y(t) \\ & = -\frac{1}{\Gamma(\alpha)}\int^{t}_{a}(t - s)^{\alpha - 1}h(s)ds + \frac{(t - a)^{\alpha - 1}}{A}\int^{b}_{a}\Big{[}\frac{(b - s)^{\alpha - 1}}{\Gamma(\alpha)} + 1\Big{]}h(s)ds \\ & + \frac{(\alpha - 1)(t - a)^{\alpha - 2}}{A}\int^{b}_{a}\Big{[}\frac{(b - s)^{\alpha - 1}}{\Gamma(\alpha)} + 1\Big{]}h(s)ds\\ & = \int^{t}_{a}\Big{[}\Big{(}\frac{(t - a)^{\alpha - 1} + (\alpha - 1)(t - a)^{\alpha - 2}}{A}\Big{)}\Big{(}\frac{(b - s)^{\alpha - 1}}{\Gamma(\alpha)} + 1\Big{)} - \frac{(t - s)^{\alpha - 1}}{\Gamma(\alpha)}\Big{]}h(s)ds \\ & + \int^{b}_{t}\Big{(}\frac{(t - a)^{\alpha - 1} + (\alpha - 1)(t - a)^{\alpha - 2}}{A}\Big{)}\Big{(}\frac{(b - s)^{\alpha - 1}}{\Gamma(\alpha)} + 1\Big{)}h(s)ds \\ & = \int^{b}_{a}G(t, s)h(s)ds.
\end{align*}
Hence the proof.
\end{proof}

Now, we prove that this Green's function is positive and obtain an upper bound for the Green's function and its integral.

\begin{thm} \label{RoF Green Positive}
The Green's function $G(t, s)$ for \eqref{RoF FDE 1} satisfies $G(t, s) > 0$ for $(t, s) \in (a, b] \times (a, b]$.
\end{thm}

\begin{proof}
For $a < t \leq s \leq b$, $$G(t, s) = \Big{(}\frac{(t - a)^{\alpha - 1} + (\alpha - 1)(t - a)^{\alpha - 2}}{A}\Big{)}\Big{(}\frac{(b - s)^{\alpha - 1}}{\Gamma(\alpha)} + 1\Big{)} > 0.$$ Now, suppose $a < s \leq t \leq b$. Consider
\begin{align} \label{G2}
\nonumber G(t, s) & = \Big{(}\frac{(t - a)^{\alpha - 1} + (\alpha - 1)(t - a)^{\alpha - 2}}{A}\Big{)}\Big{(}\frac{(b - s)^{\alpha - 1}}{\Gamma(\alpha)} + 1\Big{)} - \frac{(t - s)^{\alpha - 1}}{\Gamma(\alpha)} \\ \nonumber & = \frac{1}{A\Gamma(\alpha)}\Big{[}(t - a)^{\alpha - 1} (b - s)^{\alpha - 1} + (\alpha - 1)(t - a)^{\alpha - 2}(b - s)^{\alpha - 1} \\ \nonumber & + \Gamma(\alpha)(t - a)^{\alpha - 1} + (\alpha - 1)\Gamma(\alpha)(t - a)^{\alpha - 2}\Big{]} - \frac{(t - s)^{\alpha - 1}}{\Gamma(\alpha)} \\ \nonumber & = \frac{1}{A\Gamma(\alpha)}\Big{(}\big{[}(t - a)^{\alpha - 1}(b - s)^{\alpha - 1} - (b - a)^{\alpha - 1}(t - s)^{\alpha - 1}\big{]} \\ \nonumber & + (\alpha - 1)\big{[}(t - a)^{\alpha - 2}(b - s)^{\alpha - 1} - (b - a)^{\alpha - 2}(t - s)^{\alpha - 1}\big{]} \\ \nonumber & + \Gamma(\alpha)\big{[}(t - a)^{\alpha - 1} - (t - s)^{\alpha - 1}\big{]} + (\alpha - 1)\Gamma(\alpha)(t - a)^{\alpha - 2}\Big{)} \\ & = \frac{1}{A\Gamma(\alpha)} \big{[}E + B + C + D\big{]}.
\end{align}
Clearly, $A\Gamma(\alpha) > 0$. Consider $$(t - a)(b - s) - (b - a)(t - s) = (s - a)(b - t) \geq 0,$$ implies
\begin{equation} \label{A}
E = (t - a)^{\alpha - 1} (b - s)^{\alpha - 1} - (b - a)^{\alpha - 1} (t - s)^{\alpha - 1} \geq 0.
\end{equation}
Since $$a < s \leq t \leq b,$$ we have $$(t - a)^{\alpha - 2} \geq (b - a)^{\alpha - 2}, ~ (b - s)^{\alpha - 1} \geq (t - s)^{\alpha - 1} ~ \text{and} ~ (t - a)^{\alpha - 1} > (t - s)^{\alpha - 1},$$ implies 
\begin{align} \label{B}
\nonumber B & = (\alpha - 1)\big{[}(t - a)^{\alpha - 2}(b - s)^{\alpha - 1} - (b - a)^{\alpha - 2}(t - s)^{\alpha - 1}\big{]} \\ & \geq (\alpha - 1)(b - a)^{\alpha - 2}\big{[}(b - s)^{\alpha - 1} - (t - s)^{\alpha - 1}\big{]} \geq 0,
\end{align}
and 
\begin{equation} \label{C}
C = \Gamma(\alpha)\big{[}(t - a)^{\alpha - 1} - (t - s)^{\alpha - 1}\big{]} > 0.
\end{equation}
Clearly, 
\begin{equation} \label{D}
D = (\alpha - 1)\Gamma(\alpha)(t - a)^{\alpha - 2} > 0.
\end{equation}
Using \eqref{A} - \eqref{D} in \eqref{G2}, we have $G(t, s) > 0$. Hence the proof.
\end{proof}

\begin{thm} \label{RoF Green Max}
For the Green's function $G(t, s)$ defined in \eqref{RoF Green}, $$\max_{(t, s) \in (a, b] \times (a, b]}G(t, s) = G(t, t)$$ and $$\max_{t \in [a, b]}(t - a)^{2 - \alpha}G(t, t) = \Big{(}\frac{b - a + \alpha - 1}{A}\Big{)}\Big{(}\frac{(b - a)^{\alpha - 1}}{\Gamma(\alpha)} + 1\Big{)}.$$
\end{thm}

\begin{proof}
First, we show that for any fixed $t \in (a, b]$, $G(t, s)$ increases from $G(t, a)$ to $G(t, t)$, and then decreases to $G(t, b)$. Let $t \leq s < b$. Consider $$\frac{d}{ds}G(t, s) = -\Big{(}\frac{(t - a)^{\alpha - 1} + (\alpha - 1)(t - a)^{\alpha - 2}}{A}\Big{)}\frac{(\alpha - 1)(b - s)^{\alpha - 2}}{\Gamma(\alpha)} < 0,$$ implies $G(t, s)$ is a decreasing function of $s$. Now, suppose $a < s < t$. Consider
\begin{align} \label{G3}
\nonumber \frac{d}{ds}G(t, s) & = -\Big{(}\frac{(t - a)^{\alpha - 1} + (\alpha - 1)(t - a)^{\alpha - 2}}{A}\Big{)}\frac{(\alpha - 1)(b - s)^{\alpha - 2}}{\Gamma(\alpha)} \\ \nonumber & + \frac{(\alpha - 1)(t - s)^{\alpha - 2}}{\Gamma(\alpha)}\\ \nonumber & = \frac{(\alpha - 1)}{A\Gamma(\alpha)}\Big{[}-(t - a)^{\alpha - 1} (b - s)^{\alpha - 2} - (\alpha - 1)(t - a)^{\alpha - 2}(b - s)^{\alpha - 2} \Big{]} \\ \nonumber & + \frac{(\alpha - 1)(t - s)^{\alpha - 2}}{\Gamma(\alpha)} \\ \nonumber & = \frac{(\alpha - 1)}{A\Gamma(\alpha)}\Big{(}\big{[}-(t - a)^{\alpha - 1}(b - s)^{\alpha - 2} + (b - a)^{\alpha - 1}(t - s)^{\alpha - 2}\big{]} \\ \nonumber & + (\alpha - 1)\big{[}-(t - a)^{\alpha - 2}(b - s)^{\alpha - 2} + (b - a)^{\alpha - 2}(t - s)^{\alpha - 2}\big{]} \\ \nonumber & + \Gamma(\alpha)(t - s)^{\alpha - 2}\Big{)} \\ & = \frac{(\alpha - 1)}{A\Gamma(\alpha)} \big{[}L + M + N\big{]}.
\end{align}
Clearly, $\frac{(\alpha - 1)}{A\Gamma(\alpha)} > 0$. Consider $$(t - a)(b - s) - (b - a)(t - s) = (s - a)(b - t) \geq 0,$$ implies
\begin{equation} \label{M}
M = (\alpha - 1)\big{[}-(t - a)^{\alpha - 2}(b - s)^{\alpha - 2} + (b - a)^{\alpha - 2}(t - s)^{\alpha - 2}\big{]} \geq 0.
\end{equation}
Since $$a < s < t \leq b,$$ we have $$(t - s)^{\alpha - 2} \geq (b - s)^{\alpha - 2} ~ \text{and} ~ (b - a)^{\alpha - 1} \geq (t - a)^{\alpha - 1},$$ implies 
\begin{align} \label{L}
\nonumber L & = -(t - a)^{\alpha - 1}(b - s)^{\alpha - 2} + (b - a)^{\alpha - 1}(t - s)^{\alpha - 2} \\ & \geq (b - s)^{\alpha - 2}\big{[}-(t - a)^{\alpha - 1} + (b - a)^{\alpha - 1}\big{]} \geq 0.
\end{align}
Clearly, 
\begin{equation} \label{N}
N = \Gamma(\alpha)(t - s)^{\alpha - 2} > 0.
\end{equation}
Using \eqref{M} - \eqref{N} in \eqref{G3}, we have $G(t, s) > 0$, implies $G(t, s)$ is an increasing function of $s$. Thus, we have $$\max_{(t, s) \in (a, b] \times (a, b]}G(t, s) = G(t, t).$$ Clearly, 
\begin{align*}
\max_{t \in [a, b]}(t - a)^{2 - \alpha}G(t, t) & = \Big{(}\frac{t - a + \alpha - 1}{A}\Big{)}\Big{(}\frac{(b - t)^{\alpha - 1}}{\Gamma(\alpha)} + 1\Big{)} \\ & = \Big{(}\frac{b - a + \alpha - 1}{A}\Big{)}\Big{(}\frac{(b - a)^{\alpha - 1}}{\Gamma(\alpha)} + 1\Big{)}.
\end{align*}
\end{proof}

\begin{thm} \label{RoF Green Int Max}
For the Green's function $G(t, s)$ defined in \eqref{RoF Green}, $$\max_{t \in [a, b]}\int^{b}_{a}(t - a)^{2 - \alpha}G(t, s)ds = \Big{(}\frac{b - a + \alpha - 1}{A}\Big{)}\Big{(}\frac{(b - a)^{\alpha}}{\Gamma(\alpha + 1)} + (b - a)\Big{)}.$$
\end{thm}

\begin{proof}
Consider
\begin{align*}
& \int^{b}_{a}(t - a)^{2 - \alpha}G(t, s)ds \\ & = \int^{t}_{a}\Big{[}\Big{(}\frac{t - a + \alpha - 1}{A}\Big{)}\Big{(}\frac{(b - s)^{\alpha - 1}}{\Gamma(\alpha)} + 1\Big{)} - \frac{(t - s)^{\alpha - 1}(t - a)^{2 - \alpha}}{\Gamma(\alpha)}\Big{]}ds \\ & + \int^{b}_{t}\Big{(}\frac{t - a + \alpha - 1}{A}\Big{)}\Big{(}\frac{(b - s)^{\alpha - 1}}{\Gamma(\alpha)} + 1\Big{)}ds \\ & = \Big{(}\frac{t - a + \alpha - 1}{A}\Big{)}\Big{(}\frac{(b - a)^{\alpha}}{\Gamma(\alpha + 1)} + (b - a)\Big{)} - \frac{(t - a)^{2}}{\Gamma(\alpha + 1)}.
\end{align*}
Clearly, 
\begin{align*}
& \max_{t \in [a, b]}\Big{[}\Big{(}\frac{t - a + \alpha - 1}{A}\Big{)}\Big{(}\frac{(b - a)^{\alpha}}{\Gamma(\alpha + 1)} + (b - a)\Big{)} - \frac{(t - a)^{2}}{\Gamma(\alpha + 1)}\Big{]} \\ & = \Big{(}\frac{b - a + \alpha - 1}{A}\Big{)}\Big{(}\frac{(b - a)^{\alpha}}{\Gamma(\alpha + 1)} + (b - a)\Big{)}.
\end{align*}
Hence the proof.
\end{proof}

We are now able to formulate a Lyapunov-type inequality for the Robin boundary value problem.

\begin{thm} \label{RoF Theorem 2}
If the following fractional boundary value problem
\begin{equation} \label{RoF FDE 2}
\begin{cases}
D^{\alpha}_{a}y(t) + p(t)y(t) = 0, \quad a < t < b, \\ I^{2 - \alpha}_{a}y(a) - D^{\alpha - 1}_{a}y(a) = 0, ~ y(b) + D^{\alpha - 1}_{a}y(b) = 0,
\end{cases}
\end{equation}
has a nontrivial solution, then
\begin{equation} \label{LF Lyp}
\int^{b}_{a}(s - a)^{\alpha - 2}|p(s)|ds > \frac{\big{[}(b - a)^{\alpha - 1} + (\alpha - 1)(b - a)^{\alpha - 2} + \Gamma(\alpha)\big{]}\Gamma(\alpha)}{\big{[}(b - a)^{\alpha - 1} + \Gamma(\alpha)\big{]}(b - a + \alpha - 1)}.
\end{equation}
\end{thm}

\begin{proof}
Let $\mathfrak{B} = C_{2 - \alpha}[a, b]$ be the Banach space of functions $y$ endowed with norm $$\|y\|_{C_{2 - \alpha}} = \max_{t \in [a, b]}|(t - a)^{2 - \alpha}y(t)|.$$ It follows from Theorem \ref{RoF FDE 1} that a solution to \eqref{RoF FDE 2} satisfies the equation
\begin{equation*}
y(t) = \int^{b}_{a}G(t, s)p(s)y(s)ds.
\end{equation*}
Hence, 
\begin{align*}
\|y\|_{C_{2 - \alpha}} & = \max_{t \in [a, b]}\Big{|}(t - a)^{2 - \alpha}\int^{b}_{a}G(t, s)p(s)y(s)ds\Big{|} \\ & \leq \max_{t \in [a, b]}\Big{[}\int^{b}_{a}(t - a)^{2 - \alpha}G(t, s)|p(s)||y(s)|ds\Big{]} \\ & \leq \|y\|_{C_{2 - \alpha}} \Big{[}\max_{t \in [a, b]}\int^{b}_{a}(t - a)^{2 - \alpha}G(t, s)(s - a)^{\alpha - 2}|p(s)|ds\Big{]} \\ & \leq \|y\|_{C_{2 - \alpha}} \Big{[}\max_{t \in [a, b]}(t - a)^{2 - \alpha}G(t, t)\Big{]}\int^{b}_{a}(s - a)^{\alpha - 2}|p(s)|ds,
\end{align*}
or, equivalently, $$1 < \Big{[}\max_{t \in [a, b]}(t - a)^{2 - \alpha}G(t, t)\Big{]}\int^{b}_{a}(s - a)^{\alpha - 2}|p(s)|ds.$$ An application of Theorem \ref{RoF Green Max} yields the result. 
\end{proof}

Consider the one and two-parameter Mittag-Leffler functions \cite{KIL}
\begin{equation} \label{ML 1}
E_{\alpha}(z) = \sum^{\infty}_{k = 0}\frac{z^{k}}{\Gamma(k\alpha + 1)},
\end{equation}
and 
\begin{equation} \label{ML 2}
E_{\alpha, \beta}(z) = \sum^{\infty}_{k = 0}\frac{z^{k}}{\Gamma(k\alpha + \beta)}, 
\end{equation}
where $z$, $\beta \in \mathbb{C}$ and $\Re(\alpha) > 0$.

We use Theorem \ref{RoF Theorem 2} to obtain an interval in which the function of Mittag-Leffler functions \eqref{ML 1} and \eqref{ML 2} has no real zeros.

\begin{thm} \label{RoF Zero}
Let $1 < \alpha \leq 2$. Then, the function $E_{\alpha}(x) + (1 - x)E_{\alpha, \alpha}(x) + E_{\alpha, \alpha - 1}(x)$ has no real zeros for $$|x| \leq \frac{(\alpha - 1)(\alpha + \Gamma(\alpha))\Gamma(\alpha)}{\alpha(1 + \Gamma(\alpha))}.$$
\end{thm}

\begin{proof}
Let $a = 0$, $b = 1$ and consider the fractional boundary value problem
\begin{equation} \label{RoF FDE 3}
\begin{cases}
D^{\alpha}_{0}y(t) + \lambda y(t) = 0, \quad 0 < t < 1, \\ I^{2 - \alpha}_{0}y(0) - D^{\alpha - 1}_{0}y(0) = 0, ~ y(1) + D^{\alpha - 1}_{0}y(1) = 0.
\end{cases}
\end{equation}
By Corollary 5.1 of \cite{KIL}, the general solution of the fractional differential equation $$D^{\alpha}_{0}y(t) + \lambda y(t) = 0$$ is given by
\begin{equation} \label{GS}
y(t) = c_{1}t^{\alpha - 1}E_{\alpha, \alpha}(-\lambda t^{\alpha}) + c_{2}t^{\alpha - 2}E_{\alpha, \alpha - 1}(-\lambda t^{\alpha}), \quad t \in (0, 1].
\end{equation}
Denote by $$g(t) = t^{\alpha - 1}E_{\alpha, \alpha}(-\lambda t^{\alpha}) = t^{\alpha - 1}\sum_{n = 0}^{\infty}\frac{(-\lambda)^{n}t^{\alpha n}}{\Gamma(\alpha n + \alpha)}.$$ Then $$g'(t) = t^{\alpha - 2}\sum_{n = 0}^{\infty}\frac{(-\lambda)^{n}t^{\alpha n}}{\Gamma(\alpha n + \alpha - 1)} = t^{\alpha - 2}E_{\alpha, \alpha - 1}(-\lambda t^{\alpha}).$$ Note that
\begin{align}
I_{0}^{2 - \alpha}g(t) & = \sum_{n = 0}^{\infty}\frac{(-\lambda)^{n}}{\Gamma(\alpha n + \alpha)}I_{0}^{2 - \alpha}t^{\alpha n + \alpha - 1}\\ \notag
& = \sum_{n = 0}^{\infty}\frac{(-\lambda)^{n}}{\Gamma(\alpha n + \alpha)}\frac{\Gamma(\alpha n + \alpha)}{\Gamma(\alpha n + 2)}t^{\alpha n + 1}\\ \notag
& = t\sum_{n = 0}^{\infty}\frac{(-\lambda)^{n}}{\Gamma(\alpha n + 2)}t^{\alpha n} = tE_{\alpha, 2}(-\lambda t^{\alpha}),
\end{align}
\begin{align}
D_{0}^{\alpha - 1}g(t) & = \sum_{n = 0}^{\infty}\frac{(-\lambda)^{n}}{\Gamma(\alpha n + \alpha)}D_{0}^{\alpha - 1}t^{\alpha n + \alpha - 1}\\ \notag
& = \sum_{n = 0}^{\infty}\frac{(-\lambda)^{n}}{\Gamma(\alpha n + \alpha )}\frac{\Gamma(\alpha n + \alpha)}{\Gamma(\alpha n + 1)}t^{\alpha n} \\ \notag & = \sum_{n = 0}^{\infty}\frac{(-\lambda)^{n}}{\Gamma(\alpha n + 1)}t^{\alpha n} = E_{\alpha}(-\lambda t^{\alpha}),
\end{align}
\begin{align}
I_{0}^{2 - \alpha}g'(t) & = \sum_{n = 0}^{\infty}\frac{(-\lambda)^{n}}{\Gamma(\alpha n + \alpha - 1)}I_{0}^{2 - \alpha}t^{\alpha n + \alpha - 2}\\ \notag
& = \sum_{n = 0}^{\infty}\frac{(-\lambda)^{n}}{\Gamma(\alpha n + \alpha - 1)}\frac{\Gamma(\alpha n + \alpha - 1)}{\Gamma(\alpha n + 1)}t^{\alpha n}\\ \notag & = \sum_{n = 0}^{\infty}\frac{(-\lambda)^{n}}{\Gamma(\alpha n + 1)}t^{\alpha n} = E_{\alpha}(-\lambda t^{\alpha}),
\end{align}
and 
\begin{align}
D_{0}^{\alpha - 1}g'(t) & = \sum_{n = 0}^{\infty}\frac{(-\lambda)^{n}}{\Gamma(\alpha n + \alpha - 1)}D_{0}^{\alpha - 1}t^{\alpha n + \alpha - 2}\\ \notag
& = \sum_{n = 1}^{\infty}\frac{(-\lambda)^{n}}{\Gamma(\alpha n + \alpha - 1)}\frac{\Gamma(\alpha n + \alpha - 1)}{\Gamma(\alpha n)}t^{\alpha n - 1}\\ \notag  & = -\lambda\sum_{n = 0}^{\infty}\frac{(-\lambda)^{n}}{\Gamma(\alpha n + \alpha)}t^{\alpha n +\alpha - 1} = -\lambda g(t).
\end{align}
Also, note that
\begin{equation}
I_{0}^{2 - \alpha}g(0) = 0, ~ D_{0}^{\alpha - 1}g(0) = 1, ~ I_{0}^{2 - \alpha}g'(0) = 1, ~ D_{0}^{\alpha - 1}g'(0) = 0.
\end{equation}
Using $I^{2 - \alpha}_{0}y(0) - D^{\alpha - 1}_{0}y(0) = 0$, we get $c_{1} = c_{2}$. Using $y(1) + D^{\alpha - 1}_{0}y(1) = 0$, we get that the eigenvalues $\lambda \in \mathbb{R}$ of \eqref{RoF FDE 3} are the solutions of
\begin{equation} \label{EV}
E_{\alpha}(-\lambda) + (1 - \lambda)E_{\alpha, \alpha}(-\lambda) + E_{\alpha, \alpha - 1}(-\lambda) = 0,
\end{equation}
and the corresponding eigenfunctions are given by
\begin{equation}
y(t) = t^{\alpha - 1}E_{\alpha, \alpha}(-\lambda t^{\alpha}) + t^{\alpha - 2}E_{\alpha, \alpha - 1}(-\lambda t^{\alpha}), \quad t \in (0, 1].
\end{equation}
By Theorem \ref{RoF Theorem 2}, if a real eigenvalue $\lambda$ of \eqref{RoF FDE 3} exists, i.e. $\lambda$ is a zero of \eqref{EV}, then $$|\lambda| > \frac{(\alpha - 1)(\alpha + \Gamma(\alpha))\Gamma(\alpha)}{\alpha(1 + \Gamma(\alpha))}.$$ Hence the proof.
\end{proof}

\end{document}